\documentclass[oneside,english]{amsart}
\usepackage[T1]{fontenc}
\usepackage[latin9]{inputenc}
\usepackage{geometry}
\usepackage{amstext}
\usepackage{amsthm}
\usepackage{amssymb}

\makeatletter
\numberwithin{equation}{section}
\numberwithin{figure}{section}
\theoremstyle{plain}
\newtheorem{thm}{\protect\theoremname}
\theoremstyle{remark}
\newtheorem{rem}[thm]{\protect\remarkname}
\theoremstyle{definition}
\newtheorem{defn}[thm]{\protect\definitionname}
\theoremstyle{plain}
\newtheorem{lem}[thm]{\protect\lemmaname}
\theoremstyle{plain}
\newtheorem{prop}[thm]{\protect\propositionname}
\theoremstyle{plain}
\newtheorem{conjecture}[thm]{\protect\conjecturename}
\theoremstyle{definition}
\newtheorem{example}[thm]{\protect\examplename}

\DeclareMathOperator{\Paley}{Paley}

\makeatother

\usepackage{babel}
\providecommand{\conjecturename}{Conjecture}
\providecommand{\definitionname}{Definition}
\providecommand{\examplename}{Example}
\providecommand{\lemmaname}{Lemma}
\providecommand{\propositionname}{Proposition}
\providecommand{\remarkname}{Remark}
\providecommand{\theoremname}{Theorem}

\begin{document}
\title{Paley Graphs and S\'{a}rk\"{o}zy's Theorem In Function Fields}
\author{Eric Naslund}
\email{naslund.math@gmail.com}
\date{\today}
\begin{abstract}
S\'{a}rk\"{o}zy's theorem states that dense sets of integers must
contain two elements whose difference is a $k^{th}$ power. Following
the polynomial method breakthrough of Croot, Lev, and Pach \cite{CrootLevPachZ4},
Green proved a strong quantitative version of this result for $\mathbb{F}_{q}[T]$.
In this paper we provide a lower bound for S\'{a}rk\"{o}zy's theorem
in function fields by adapting Ruzsa's construction \cite{Ruzsa1984DifferenceSetsWithoutSquares}
for the analogous problem in $\mathbb{Z}$. We construct a set $A$
of polynomials of degree $<n$ such that $A$ does not contain a $k^{th}$
power difference with $|A|=q^{n-n/2k}$. Additionally, we prove a
handful of results concerning the independence number of generalized
Paley Graphs, including a generalization of a claim of Ruzsa, which
helps with understanding the limit of the method.
\end{abstract}

\maketitle

\section{Introduction}

In 1978 S\'{a}rk\"{o}zy proved that if $A\subset\mathbb{Z}$ is a
dense set of integers then it contains a $k^{th}$-power difference
for any $k\geq2$ \cite{Sarkozy1978SarkozysTheorem}. For a set $A\subset\{1,\dots,N\}$
that avoids $k^{th}$-power differences the best quantitative upper
bound, due to Pintz, Steiger, and Szemer\'{e}di \cite{PintzSteigerSzemeredi1988BestBoundsSarkozysTheorem},
takes the form 
\[
|A|\leq N^{1-o(1)},
\]
and the best lower bound, due to Ruzsa \cite{Ruzsa1984DifferenceSetsWithoutSquares}
(improved in some cases by Lewko \cite{Lewko2015ImprovedSarkozyLowerBounds}),
is of the form $N^{1-c}$ where $c\leq\frac{1-\delta}{k}$ for some
small fixed $\delta>0$. Following recent advances in the polynomial
method \cite{CrootLevPachZ4,EllenbergGijswijtCapsets,TaosBlogCapsets},
Green \cite{Green2017SarkozyTheoremInFunctionFields} gave a strong
quantitative result for the function field analogue of S\'{a}rk\"{o}zy's
theorem. Let $P_{q,n}$ denote the space of polynomials in the variable
$T$ over $\mathbb{F}_{q}$ of degree $<n$. For $k\geq2$, Green
proved that any subset $A\subset P_{q,n}$ that does not contain two
distinct polynomials $u(T),v(T)$ such that $u(T)-v(T)=b(T)^{k}$
for some $b\in\mathbb{F}_{q}[T]$ has size at most 
\[
|A|\leq q^{n(1-c_{k})}
\]
where
\[
c_{k}=\frac{1}{2k^{2}D_{q}(k)^{2}\log q}
\]
and $D_{q}(k)$ denotes the sum of digits of $k$ in base $q$. In
this paper, by adapting Ruzsa's construction in $\mathbb{Z}$, we
prove the following lower bound for the function field setting:
\begin{thm}
\label{thm:sarkozy_kth_power_lower}Let $k\geq2$, and suppose that
$\gcd(k,q-1)>1$. For $n\equiv0\pmod{2k}$ there exists a set $A\subset P_{q,n}$
of size 
\[
|A|=q^{n\left(1-\frac{1}{2k}\right)}
\]
 that does not contain a $k^{th}$ power difference. 
\end{thm}

The construction results in a bound for all $n$, since $P_{q,n}\subset P_{q,n+1}$.
In section \ref{sec:Generalized-Paley-Graphs}, we prove a handful
of results concerning the independence number of generalized Paley
graphs, including a generalization of a claim of Ruzsa. We prove a
basic lower bound for the independence number of a product, which
is ingredient in the proof of Theorem \ref{thm:sarkozy_kth_power_lower},
and we also prove upper bounds to help understand the limits of the
method. For $k\geq3$, Theorem \ref{thm:sarkozy_kth_power_lower},
could potentially be improved by finding larger independent sets in
products of generalized Paley graphs. In subsection \ref{subsec:Limits-of-the-method}
we give a case with an improved bound, and discuss the limit of the
method. When $k=2$, we conjecture that the lower bound achieved,
$q^{\frac{3}{4}n}$, is optimal.
\begin{rem}
For $\gcd(k,q-1)=1$, every element of $\mathbb{F}_{q}$ is a $k^{th}$
power, and it is not always possible to obtain a lower bound as strong
as Theorem \ref{thm:sarkozy_kth_power_lower}. For $k=q^{r}$, a pigeonhole
argument achieves the upper bound
\[
|A|\leq q^{n-\lfloor\frac{n-1}{k}\rfloor}.
\]
For any $k$, there are at $q^{1+\lfloor\frac{n-1}{k}\rfloor}$ distinct
$k^{th}$ powers with degree $\leq n-1$, and so the greedy construction
yields a set $A$ with no $k^{th}$ power differences of size
\begin{equation}
|A|=q^{n-1-\lfloor\frac{n-1}{k}\rfloor},\label{eq:greedy_lower_bound}
\end{equation}
and hence in the case of $k=q^{r}$, this upper bound is tight within
a factor of $q$.
\end{rem}

\subsection{Graph Theoretic Notation}

We will use the following graph theoretic notation throughout the
paper: For a directed graph $G=(V,E)$, let $\alpha(G)$ denote the
size of the largest independent set, that is the largest set such
that no two vertices are connected by a directed edge
\[
\alpha(G)=\max_{W\subset V}\left\{ |W|:\ \forall\ x,y\in W,\ (x,y)\notin E\right\} .
\]
We let $\omega(G)$ denote the clique number, which equals $\alpha(\overline{G})$
where $\overline{G}$ is the complement graph. Here the complement
graph is $\overline{G}=(V,\overline{E})$ where $\overline{E}=\left\{ (x,y)|\ (x,y)\in V^{2},\ x\neq y,\ (x,y)\notin E\right\} $.

\subsubsection*{The Strong Graph Product:}

Given two graphs $G=(V,E)$, $G=(V',E')$, the strong graph product\emph{
}of $G$ and $G'$ is the graph $G\boxtimes G'$ with vertex set $V\times V'$,
and edge set defined by connecting $(v,v'),\ (u,u')\in V\times V'$
if $(v,u)\in E$ and $(v',u')\in E'$, or $(v,u)\in E$ and $v'=u'$,
or $v=u$ and $(v',u')\in E'$. The strong graph product as defined
applies to both directed and undirected graphs.

\subsubsection*{Shannon Capacity:}

For a graph $G$ and $n\geq2$, we let $G^{\boxtimes n}$ denote the
$n$-fold strong graph product of $G$ with itself. The Shannon Capacity
of an undirected graph $G$ is defined to be
\[
\Theta(G)=\limsup_{n\rightarrow\infty}\left(\alpha(G^{\boxtimes n})\right)^{\frac{1}{n}}.
\]

\subsubsection*{Lov\'{a}sz Theta Function:}

For an undirected graph $G$, the Lov\'{a}sz Theta Function, $\vartheta(G)$,
is a minimization over orthonormal representations of $G$, see \cite{Lovasz1979ShannonCapacityOfAGraph}
for a precise definition. In particular, for undirected $G,H$, $\vartheta$
satisfies
\begin{align}
\Theta(G) & \leq\vartheta(G),\label{eq:Lovasz_Shannon_bound}\\
\vartheta(G\boxtimes H) & =\vartheta(G)\vartheta(H),
\end{align}
and if $G$ is vertex transitive, then 
\begin{equation}
\vartheta(G)\vartheta(\overline{G})=|G|.\label{eq:lovasz_vertex_transitive_product_property}
\end{equation}

\section{Generalized Paley Graphs\label{sec:Generalized-Paley-Graphs}}

Generalized Paley Graphs were introduced by Cohen \cite{Cohen1988CliqueNumberOfPaleyGraphs}
and reintroduced by Lim and Praeger \cite{LimPraeger2009GeneralizedPaleyGraphs}.
We expand their definition slightly to allow for the vertex set to
be a ring rather than just a field, as this will be relevant later
(see Theorem \ref{thm:Ruzsa_generalization_odd}).
\begin{defn}
\label{def:Generalized_Paley_Graph_Definition}For a finite commutative
ring $R$, let $\Paley_{k}\left(R\right)$ be the (possibly directed)
graph with vertex set $V=R$, and edge set 
\[
E=\left\{ (x,y):\ x-y=z^{k}\text{ for some }z\in R\right\} .
\]
\end{defn}

With this notation, for $q\equiv1\pmod{4}$, $\Paley_{2}(\mathbb{F}_{q})$
is the usual Paley graph. The graph $\Paley_{k}(\mathbb{F}_{q})$,
where two elements are connected if they differ by a $k^{th}$ power,
is undirected if and only if $\frac{q-1}{\gcd(q-1,k)}$ is even, since
in this case $-1$ is a $k^{th}$ power. If $\gcd(k,q-1)=1$, then
every element of $\mathbb{F}_{q}$ is a $k^{th}$ power, and so $\Paley_{k}(\mathbb{F}_{q})$
is complete graph on $q$ vertices, and if $\gcd(k,q-1)=d<k$, then
replacing $k$ with $d$ does not change the graph. The assumption
$q\equiv1\pmod{2k}$ is often used as it assures both that the graph
is undirected and that $k^{th}$ powers are indeed relevant. 

Ruzsa \cite{Ruzsa1984DifferenceSetsWithoutSquares} gave a lower bound
for S\'{a}rk\"{o}zy's theorem in $\mathbb{Z}$ based on a maximization
involving the quantity
\[
\alpha\left(\Paley_{k}\left(\mathbb{Z}/m\mathbb{Z}\right)\right)
\]
 for squarefree $m$. In Section \ref{sec:The-Lower-Bound-Construction}
we give a lower bound for the function field case in terms of the
independence number of a product of this graph. Our results are most
naturally stated in terms of the independence number of products of
generalized Payley graphs, and so we define 

\begin{equation}
r_{k,n}(R)=\begin{cases}
\alpha\left(\Paley_{k}(R)\right) & \text{ when }n=1\\
\alpha\left(\Paley_{k}(R)^{\boxtimes n}\right) & \text{ when }n\geq2
\end{cases}.\label{eq:r_k_definition}
\end{equation}
For products of rings, these graphs can be factored. For composite
$m$, we have the following lemma:
\begin{lem}
\label{lem:paley_direct_product}Let $n,m>1$ be relatively prime.
Then 
\[
\Paley_{k}(\mathbb{Z}/mn\mathbb{Z})=\Paley_{k}(\mathbb{Z}/m\mathbb{Z})\boxtimes\Paley_{k}(\mathbb{Z}/n\mathbb{Z}).
\]
\end{lem}

\begin{proof}
This follows from the Chinese Remainder Theorem and the fact that
an element is a $k^{th}$-power in $\mathbb{Z}/mn\mathbb{Z}$ if and
only if it maps to a $k^{th}$ power in both $\mathbb{Z}/m\mathbb{Z}$
and $\mathbb{Z}/n\mathbb{Z}$.
\end{proof}

\subsection{The Independence Number\label{subsec:The-Independence-Number-general-payley}}

For $R=\mathbb{F}_{q}$, the clique number\emph{ }of the generalized
Payley graph is a well studied quantity. See Yip's Masters Thesis
\cite{Yip2021ThesisCliquesInGeneralizedPaleyGraphs,Yip2022CliqueNumberPrimePower}
for a discussion of lower and upper bounds. In particular, when the
graph is undirected, the best lower bound is due to Cohen \cite{Cohen1988CliqueNumberOfPaleyGraphs}
\begin{equation}
\frac{p}{(p-1)\log d}\left(\frac{1}{2}\log q-2\log\log q\right)-1\leq\omega\left(\Paley_{k}(\mathbb{F}_{q})\right)=\alpha\left(\overline{\Paley_{k}(\mathbb{F}_{q})}\right).\label{eq:cohen_lower_bound}
\end{equation}
 When $k=2$, the graph $\Paley_{k}(\mathbb{F}_{q})$ is self-complementary,
and so $\omega(\Paley_{2}(\mathbb{F}_{q}))=\alpha(\Paley_{2}(\mathbb{F}_{q}))$
but for $k>2$, $\Paley_{k}(\mathbb{F}_{q})$ is isomorphic to a subgraph
of $\overline{\Paley_{k}(\mathbb{F}_{q})}$, and hence
\begin{equation}
\omega\left(\Paley_{k}(\mathbb{F}_{q})\right)=\alpha\left(\overline{\Paley_{k}(\mathbb{F}_{q})}\right)\leq\alpha\left(\Paley_{k}(\mathbb{F}_{q})\right),\label{eq:independence_number_complement_bound}
\end{equation}
that is the lower bounds for the clique number imply lower bounds
for the independence number. In some unique cases, such as when $q=p^{s}$
and $k|\frac{q-1}{p^{r}-1}$ for some $r|s$, there are significantly
stronger lower bounds for the size of the clique in $\Paley_{k}(\mathbb{F}_{q})$,
see \cite{BroereDomanRidley1988largeCliquesInSomePayleyGraphs,Yip2021ThesisCliquesInGeneralizedPaleyGraphs}
for more details. The upper bound for the clique number is $\sqrt{q}$,
and for $k=2$ this was recently improved in \cite{HansonPetridis2021CliqueNumberOfPaleyGraphs}.

The case for the independence number is different however, as in some
cases it grows above $\sqrt{q}$ for larger $k$. In the first non-trivial
case, when $k=3$ and $q=7$, the Paley Graph is precisely the $7$-cycle,
and we have an independent set of size $3=7^{0.5645...}$. Indeed,
if $q=1+2k$, then the only $k^{th}$ powers in $\mathbb{F}_{q}$
are $-1,0,1$, and hence $\Paley_{k}(\mathbb{F}_{q})$ is isomorphic
to $C_{2k+1}$, the $(2k+1)$-cycle. This contains an independent
set of size $k$, and so there are infinitely many graphs satisfying
\begin{equation}
q^{1-\frac{\log(2)}{\log(2k+1)}}\leq\alpha\left(\Paley_{k}\left(\mathbb{F}_{q}\right)\right).\label{eq:sparse_lower_bound}
\end{equation}
The following theorem gives a basic upper bound for the independence
number of these graphs.
\begin{thm}
\label{thm:payley_upper_bound}Suppose that $-1$ is a $k^{th}$ power
in $\mathbb{F}_{q}$. Then we have that
\begin{equation}
\vartheta\left(\Paley_{k}(\mathbb{F}_{q})\right)\leq q^{1-\frac{1}{k}}\label{eq:lovasz_theta_bound}
\end{equation}
where $\vartheta$ is the Lov\'{a}sz Theta Function.
\end{thm}

In particular, by (\ref{eq:Lovasz_Shannon_bound}) and the definition
of $\Theta$ as a limsup, it follows that

\begin{equation}
r_{k,2}(\mathbb{F}_{q})\leq q^{2\left(1-\frac{1}{k}\right)}.\label{eq:simple_r_k_2_upper_bound}
\end{equation}
To prove this theorem, we need to prove the following fact about the
complement graph:
\begin{lem}
\label{lem:shannon_capacity_lower_bound}We have that 
\[
\alpha\left(\overline{\Paley_{k}(\mathbb{F}_{q})}^{\boxtimes k}\right)\geq q,
\]
and hence when the graph is undirected 
\begin{equation}
\vartheta\left(\overline{\Paley_{k}(\mathbb{F}_{q})}\right)\geq q^{\frac{1}{k}}.\label{eq:lower_bound_lovasz_theta}
\end{equation}
\end{lem}

\begin{proof}
If $\gcd(k,q-1)=1$, then $\overline{\Paley_{k}(\mathbb{F}_{q})}$
is the totally isolated graph, the complement of the complete graph,
and so the result holds trivially. Assume that $\gcd(k,q-1)>1$, and
let $\beta\in\mathbb{F}_{q}$ be a cyclic generator of $\mathbb{F}_{q}^{*}$.
Consider the set 
\[
A=\left\{ (x,\beta x,\beta^{2}x,\dots,\beta^{k-1}x)|\ x\in\mathbb{F}_{q}\right\} .
\]
Let $x,y\in\mathbb{F}_{q}$ be two elements with $x-y\neq0$, and
write $(x-y)=\beta^{a}$ for some $a$. Then $(x-y)\beta^{j}$ will
be a $k^{th}$ power for the value of $j\in\{0,1,\dots,k-1\}$ satisfying
$j\equiv-a\pmod k$. This proves that $A$ is an independent set.
Equation (\ref{eq:lower_bound_lovasz_theta}) then follows from (\ref{eq:lovasz_theta_bound})
since the Lov\'{a}sz Theta Function upper bounds the size of the
largest independent set.
\end{proof}
\begin{proof}
(of Theorem \ref{thm:payley_upper_bound}). Since the Generalized
Payley graph is vertex transitive, and undirected since $-1$ is a
$k^{th}$ power, by (\ref{eq:lovasz_vertex_transitive_product_property})
\[
\vartheta\left(\Paley_{k}(\mathbb{F}_{q})\right)\vartheta\left(\overline{\Paley_{k}(\mathbb{F}_{q})}\right)=q,
\]
Equation (\ref{eq:lovasz_theta_bound}) then follows from this and
the lower bound in Lemma \ref{lem:shannon_capacity_lower_bound}.
\end{proof}
The following proposition gives a lower bound for the independence
number of a product, which will be used in the next section in the
proof of Theorem \ref{thm:sarkozy_kth_power_lower}.
\begin{prop}
\label{prop:rk_2_lower_bound}For $\gcd(k,q-1)>1$ we have that $r_{k,2}(\mathbb{F}_{q})\geq q.$
In particular, when $k=2$ and $q$ is odd, $r_{2,2}(\mathbb{F}_{q})=q$.
\end{prop}

\begin{proof}
Since $\gcd(k,q-1)>1$, there exists $\beta\in\mathbb{F}_{q}$ that
is not a $k^{th}$ power. Then $A=\left\{ (x,\beta x)\ |x\in\mathbb{F}_{q}\right\} $
is an independent set in $\Paley_{k}(\mathbb{F}_{q})\boxtimes\Paley_{k}(\mathbb{F}_{q})$,
since for any distinct $x,y\in\mathbb{F}_{q}$, only one of $x-y$
and $\beta(x-y)$ can be a $k^{th}$ power. The final statement follows
from the upper bound \ref{eq:simple_r_k_2_upper_bound}.
\end{proof}
Proposition \ref{eq:lower_bound_lovasz_theta} and Theorem \ref{thm:payley_upper_bound}
together imply that 
\begin{equation}
q^{\frac{1}{2}}\leq\Theta\left(\Paley_{k}(\mathbb{F}_{q})\right)\leq q^{1-\frac{1}{k}}\label{eq:paley_shannon_capacity}
\end{equation}
and 
\begin{equation}
q^{\frac{1}{k}}\leq\Theta\left(\overline{\Paley_{k}(\mathbb{F}_{q})}\right)\leq q^{\frac{1}{2}}.\label{eq:paley_complement_shannon_capacity}
\end{equation}
For $k\geq3$, for prime fields, we believe that neither of these
inequalities are sharp.
\begin{conjecture}
For $k\geq3$, there exists $a_{k},b_{k}>0$ such that for any prime
$p\equiv1\pmod{k}$. 
\[
p^{\frac{1}{2}+a_{k}}\leq\Theta\left(\Paley_{k}(\mathbb{F}_{p})\right)\leq p^{1-\frac{1}{k}-b_{k}}.
\]
\end{conjecture}

\subsection{Composite $m$}

We conclude this section with a generalization of a result of Ruzsa
from \cite{Ruzsa1984DifferenceSetsWithoutSquares}. For $k=2$, the
following Theorem was proven by Ruzsa, but the proof was not published. 
\begin{thm}
\label{thm:Ruzsa_generalization_odd}Let $m>1$ be squarefree, let
$k=d2^{s}$ where $d$ is odd, and suppose that each prime dividing
$m$ is of the form $p\equiv1\pmod{2^{s+1}}$. Then if $A\subset\mathbb{Z}/m\mathbb{Z}$
does not contain two elements whose difference is a $k^{th}$-power
we have 
\[
|A|<m^{1-\frac{1}{k}}.
\]
\end{thm}

\begin{proof}
We will prove that 
\[
\vartheta\left(\Paley_{k}\left(\mathbb{Z}/m\mathbb{Z}\right)\right)<m^{1-\frac{1}{k}},
\]
which implies the result since $\alpha(G)\leq\vartheta(G)$ for any
$G$. Let $m=p_{1}\cdots p_{r}$. Then by Lemma \ref{lem:paley_direct_product}
\[
\Paley_{k}\left(\mathbb{Z}/m\mathbb{Z}\right)=\Paley_{k}\left(\mathbb{Z}/p_{1}\mathbb{Z}\right)\boxtimes\cdots\boxtimes\Paley_{k}\left(\mathbb{Z}/p_{r}\mathbb{Z}\right).
\]
The condition $p_{i}\equiv1\pmod{2^{s+1}}$ for each $p_{i}|m$ guarantees
that $\frac{p_{i}-1}{\gcd(p_{i}-1,k)}$ will be even, and hence that
$-1$ is a $k^{th}$-power in $\mathbb{Z}/p_{i}\mathbb{Z}$, and so
these graphs are undirected. The multiplicative property of the Lov\'{a}sz
Theta Function \cite[Lemma 2]{Lovasz1979ShannonCapacityOfAGraph}
implies that 
\[
\vartheta\left(\Paley_{k}\left(\mathbb{Z}/m\mathbb{Z}\right)\right)\leq\prod_{i=1}^{r}\vartheta\left(\Paley_{k}\left(\mathbb{Z}/p_{i}\mathbb{Z}\right)\right),
\]
and hence by equation (\ref{eq:lovasz_theta_bound})
\[
\vartheta\left(\Paley_{k}\left(\mathbb{Z}/m\mathbb{Z}\right)\right)\leq\prod_{i=1}^{r}p_{i}^{1-\frac{1}{k}}=m^{1-\frac{1}{k}}.
\]
The inequality can be made strict since the left hand side is an integer,
but the right hand side is not.
\end{proof}
When $k$ is odd, then the bound holds for any odd squarefree integer
$m$. For $k=3$, Matolcsi and Ruzsa recently gave the superior upper
bound $O_{\epsilon}(m^{\frac{1}{2}+\epsilon})$ for squarefree $m$
\cite{MatolcsiRuzsa2021DifferenceSetsOfCubesModuloM} using methods
from \cite{MatolcsiRuzsa2014DifferenceSetsExponentialSumsI}. In the
case where $m$ is a product of primes that are not necessarily of
the form $p\equiv1\pmod{2^{s+1}}$, it seems likely that the same
upper bound holds, however for the $k=2$ proving this seems more
complex, see \cite{FordGabdullin2021SetsAvoidingSquareDifferencesModuloM}
for more details. Should such a result hold for all $m$, then Ruzsa's
lower bound construction in $[1,N]$ for sets without a $k^{th}$
power difference cannot yield a set of size greater than $N^{1-1/k^{2}}$.

\section{The Lower Bound Construction\label{sec:The-Lower-Bound-Construction}}

To prove Theorem \ref{thm:sarkozy_kth_power_lower}, we prove a general
lower bound in terms of $r_{k,2}(\mathbb{F}_{q})$, and then apply
Proposition \ref{prop:rk_2_lower_bound}. We conclude this section
with a discussion of the limits of this method.

\subsection{Main Result}
\begin{thm}
\label{thm:sarkozy_general_paley_lower}Suppose that $n\equiv0\ (2k)$
and let $F\in\mathbb{F}_{q}[T]$ be a polynomial of degree $k$. Then
there exists a set $A\subset P_{q,n}$ of size 
\begin{equation}
|A|\geq\left(r_{k,1}(\mathbb{F}_{q})\right)^{\frac{n}{k}}q^{n\left(1-\frac{1}{k}\right)}\label{eq:sarkozy_general_poly}
\end{equation}
that does not contain $p,p'$ such that $p'-p=F(u)$ for some $u\in\mathbb{F}_{q}[T]$.
If $F(T)=b_{k}T^{k}$ for $b_{k}\neq0$, then we have the improved
bound
\begin{equation}
|A|\geq\left(r_{k,2}(\mathbb{F}_{q})\right)^{\frac{n}{2k}}q^{n\left(1-\frac{1}{k}\right)}.\label{eq:sarkozy_general_kth_power}
\end{equation}
\end{thm}

\begin{proof}
Suppose that $n\equiv0\ (2k)$, and let $S\subset\mathbb{F}_{q}$
be a maximal independent set in $\Paley_{k}(\mathbb{F}_{q})$ that
contains $0$, and let $b_{k}u^{k}$ be the first coefficient of $F(u)$.
Let $A\subset P_{q,n}$ be the set of polynomials of the form 
\[
c_{0}+c_{1}T+c_{2}T^{2}+\cdots+c_{n-2}T^{n-2}+c_{n-1}T^{n-1}
\]
where
\[
\begin{cases}
c_{i}\in\mathbb{F}_{q} & \text{ when }i\not\equiv0\ (k)\\
c_{i}\in b_{k}S & \text{ when }i\equiv0\ (k)
\end{cases}.
\]
The first non-zero coefficient of the difference of two elements in
$A$ will either be a power that is not divisible by $k$, or will
equal $b_{k}(s-s')T^{jk}$ for some $j$ where $s,s'\in S$. For $u=c_{j}T^{j}+\cdots+c_{0}$,
the first coefficient of $F(u)$ will equal $b_{k}(c_{j}T^{j})^{k}$,
but since $s-s'$ is never a $k^{th}$ power by definition of $S$,
this can never be of the form $b_{k}(s-s')T^{jk}$ for any $u\in P_{q,n}$.
This proves (\ref{eq:sarkozy_general_poly}) since $|S|=r_{k,1}(\mathbb{F}_{q})$.
In the case where $F(u)=u^{k}$, we can improve the bound by making
use of both the first and last coefficient. Let $U\subset\mathbb{F}_{q}\times\mathbb{F}_{q}$
be an independent set in $\Paley_{k}(\mathbb{F}_{q})\boxtimes\Paley_{k}(\mathbb{F}_{q})$,
and consider the set $A$ of polynomials of the form 
\[
c_{0}+c_{1}T+c_{2}T^{2}+\cdots+c_{n-2}T^{n-2}+c_{n-1}T^{n-1}
\]
 where

\[
\begin{cases}
c_{i}\in\mathbb{F}_{q} & \text{ when }i\not\equiv0\ (k)\\
(c_{i},c_{n-k-i})\in U & \text{ when }i\equiv0\ (k),\ i<\frac{n}{2}
\end{cases}.
\]
Note that since $n\equiv0\pmod{2k}$, $T^{n-k}$ is the $k^{th}$-power
with the largest degree in $P_{q,n}$. Suppose that $u$ is a difference
of two elements of $A$, and write $u=\sum_{i=0}^{n-1}a_{i}T^{i}$
for coefficients $a_{i}$, some of which may equal $0$. Let $j$
be the index of the non-zero coefficient $a_{j}$ whose degree is
farthest from the middle, that is, let $j$ be such that $a_{j}\neq0$,
and $|j-\frac{n-k}{2}|$ is maximal. In the event of a tie between
$2$ non-zero coefficients, take $j>\frac{n-k}{2}$. If $j\not\equiv0\ \pmod{k}$,
then $u$ cannot be a $k^{th}$-power. If $k|j$, consider the pair
of coefficients $(a_{j},a_{n-k-j})$, and assume without loss of generality
that $j>\frac{n-k}{2}$. By definition of $j$, we have that 
\[
a_{i}=\begin{cases}
0 & \text{if }i>j\\
0 & \text{if }i<n-k-j
\end{cases},
\]
and so $a_{j}$ and $a_{n-k-j}$ must both be $k^{th}$-powers for
$u$ to be a $k^{th}$-power. Note that $0$ is a $k^{th}$-power,
and $a_{n-k-j}$ could possibly equal $0$. Since $a_{j}\neq0$, and
since $(a_{j},a_{n-k-j})\in U-U$, by definition of $U$ at least
one of $a_{j},a_{n-k-j}$ is not a $k^{th}$-power. This implies that
$u$ is not $k^{th}$-power, and hence $A$ contains no $k^{th}$-power
differences. Since $|U|=r_{k,2}(\mathbb{F}_{q})$, we have that
\[
|A|\geq\left(r_{k,2}(\mathbb{F}_{q})\right)^{\frac{n}{2k}}q^{n\left(1-\frac{1}{k}\right)}.
\]
The result follows for $F(T)=b_{k}T^{k}$ for $b_{k}\neq0$ by multiplying
the elements of $U$ by $b_{k}$ in the construction.
\end{proof}
Equation (\ref{eq:cohen_lower_bound}), Cohen's clique number lower
bound, yields a nontrivial bound in (\ref{eq:sarkozy_general_poly})
for sets avoiding a general polynomial $F$. Theorem \ref{thm:sarkozy_kth_power_lower}
follows immediately from equation (\ref{eq:sarkozy_general_kth_power})
and Proposition \ref{prop:rk_2_lower_bound}.
\begin{rem}
Theorem \ref{thm:sarkozy_general_paley_lower} is similar to Ruzsa's
lower bound in the integers. Let $n=\log_{m}N$. Ruzsa proved that
there exists $A\subset[1,x]$ that contains no two elements whose
difference is a $k^{th}$ power satisfying
\[
|A|\geq\frac{1}{m}\left(r_{k,1}(\mathbb{Z}/m\mathbb{Z})\right)^{\frac{n}{k}}N^{\left(1-\frac{1}{k}\right)}.
\]
(The quantitative result was improved by Lewko \cite{Lewko2015ImprovedSarkozyLowerBounds}
based on the fact that $\log r_{k,1}(\mathbb{Z}/m\mathbb{Z})/\log m$
is larger for $m=205$ than for $m=65$. See also \cite{Younis2019PolynomialSzemerediLowerBounds}
for an improved bound for a related problem). The bound for $k^{th}$
powers in Theorem \ref{thm:sarkozy_kth_power_lower} is similar, but
utilizes the fact that there is no ``overflow'' when taking powers
of a polynomial in $\mathbb{F}_{q}$ so both the first and last coefficients
play a role instead of only the last coefficient. This results in
a better bound for function fields since we always have 
\[
\left(r_{k,2}(\mathbb{F}_{q})\right)^{\frac{1}{2}}\geq r_{k}(\mathbb{F}_{q}).
\]
\end{rem}

\subsection{\label{subsec:Limits-of-the-method}Limits of the Method}

Theorem \ref{thm:payley_upper_bound} implies that one must use a
different method to improve the lower bound in Theorem \ref{thm:sarkozy_kth_power_lower}
beyond
\begin{equation}
|A|=q^{n\left(1-\frac{1}{k^{2}}\right)}.\label{eq:best_lower_bound_possible}
\end{equation}
One can ask if the lower~bound in proposition \ref{prop:rk_2_lower_bound}
is optimal. This turns out not to be true in general for $k\geq3$,
as described in the comments preceeding equation (\ref{eq:sparse_lower_bound}).
When $p=2k+1$ is a prime, $\Paley_{k}(\mathbb{F}_{p})$ is isomorphic
to $C_{2k+1}$, the $(2k+1)$-cycle, and the largest independent set
in $C_{2k+1}\boxtimes C_{2k+1}$ has size $k^{2}+\lfloor\frac{k}{2}\rfloor$
\cite[Theorem 7.1]{Hales1973CycleProductOfCyclesDimension2IndependenceNumber}.
Hence, in this case 
\[
r_{k,2}(\mathbb{F}_{p})=k^{2}+\biggr\lfloor\frac{k}{2}\biggr\rfloor.
\]
One can easily verify that in this case this construction results
in a lower bound better than Theorem \ref{thm:sarkozy_kth_power_lower}
but weaker than the best possible from Equation (\ref{eq:best_lower_bound_possible}).
The following example helps illustrate the size of the gap between
the upper and lower bounds with an explicit case.
\begin{example}
\label{exa:specific_c7_example}Consider the specific case of $k=3$
and $q=7$. The most precise upper bound obtained by Green's method,
where we calculate the value of the minimum instead of using a Chernoff
bound, is
\[
2\cdot\left(\min_{0<t<1}\frac{1-t^{7}}{(1-t)t^{6\cdot\frac{4}{9}}}\right)^{n}=2\cdot\left(6.903\dots\right)^{n}.
\]
Theorem \ref{thm:sarkozy_kth_power_lower} gives the lower bound of
$7^{\frac{5}{6}n}=\left(5.061\dots\right)^{n}$. Since $\Paley_{3}(\mathbb{F}_{7})$
is precisely $C_{7}$, the $7$-cycle, using the fact that 
\[
\alpha\left(C_{7}\boxtimes C_{7}\right)=10,
\]
Theorem \ref{thm:sarkozy_general_paley_lower} yields the improved
lower bound 
\[
\left(10^{\frac{1}{6}}7^{\frac{2}{3}}\right)^{n}=\left(5.371\dots\right)^{n},
\]
which is the limit of the method in this case. There is still a considerable
gap between the upper bound and the best possible lower bound this
method can produce. 
\end{example}

Given the gap between upper and lower bounds, we may ask which is
closer to the truth? We believe that in the Function Field setting
the lower bound is close to the truth, and conjecture that it is exact
when $k=2$ (See \cite[Section 1.4]{Rice2019MaximalExtensionFurstenburgSarkozyDiscussion}
for speculation on the integer setting).
\begin{conjecture}
\label{conj:Lower-bound-conjecture}Let $k\geq2$, and suppose that
$\gcd(k,q-1)>1$. For $n\equiv0\ (2k)$, any set $A\subset P_{q,n}$
that does not contain a $k^{th}$ power difference has size at most
\[
|A|\leq q^{n\left(1-\frac{1}{k^{2}}\right)}.
\]
In particular, for $k=2$ and $q$ odd, we conjecture that Theorem
\ref{thm:sarkozy_kth_power_lower} is tight.
\end{conjecture}

\specialsection*{Acknowledgements}

I would like to thank Will Sawin for comments that simplified the
proof of Proposition \ref{prop:rk_2_lower_bound}. I would also like
to thank the anonymous referee for their helpful comments and corrections.

\bibliographystyle{plain}

\end{document}